\documentclass[12pt]{article}
\usepackage[reqno]{amsmath} 
\usepackage{amssymb,amsthm} 

\newcommand\al{\alpha}
\newcommand\be{\beta}

\newcommand\la{\lambda}
\newcommand\La{\Lambda}
\newcommand\cx{{\mathbb C}}
\newcommand\re{{\mathbb R}} 
\newcommand\ints{{\mathbb Z}} 
\newcommand\sbs{\subseteq}
\newcommand\ff{{\mathbb F}}

\newcommand\opk[1]{\mathop{\mathrm{#1}}\nolimits}

\newcommand\Hom{\opk{Hom}}

\newcommand\spn{\opk{span}}

\newcommand\tr{\opk{tr}}
\newcommand\GL{\opk{GL}}
\newcommand\U{\opk{U{}}}
\newcommand\PU{\opk{PU{}}}

\renewcommand\Re{\opk{Re}}
\renewcommand\Im{\opk{Im}}

\newcommand\conj[1]{\overline{#1}}
\newcommand\ip[2]{\left\langle#1,#2\right\rangle}
\newcommand\abs[1]{\left|#1\right|}

\newcommand\HomU[3]{\Hom(\U(#1),#2,#3)}

\newcommand\dimHomU[3]{D(#1,#2,#3)}

\theoremstyle{plain}
\newtheorem{theorem}{Theorem}
\newtheorem{lemma}[theorem]{Lemma}
\newtheorem{corollary}[theorem]{Corollary}
\newcommand\defn[1]{\textsl{#1}}

\begin{document}

\title{Unitary designs and codes}

\author{Aidan Roy\footnote{email: aroy@qis.ucalgary.ca} \\ {\small Department of Mathematics and Statistics, and,} \\ {\small Institute for Quantum Information Science,} \\ {\small University of Calgary, Calgary, Alberta T2N 1N4, Canada.} \\ \\
A. J. Scott\footnote{email: andrew.scott@griffith.edu.au} \\ {\small Centre for Quantum Dynamics, and,} \\ {\small Centre for Quantum Computer Technology,} \\ {\small Griffith University, Brisbane, Queensland 4111, Australia.}}


\date{}

\maketitle

\begin{abstract}
A unitary design is a collection of unitary matrices that approximate the entire unitary group, much like a spherical design approximates the entire unit sphere. In this paper, we use irreducible representations of the unitary group to find a general lower bound on the size of a unitary $t$-design in $\U(d)$, for any $d$ and $t$. We also introduce the notion of a unitary code --- a subset of $\U(d)$ in which the trace inner product of any pair of matrices is restricted to only a small number of distinct values --- and give an upper bound for the size of a code of degree $s$ in $\U(d)$ for any $d$ and $s$. These bounds can be strengthened when the particular inner product values that occur in the code or design are known. Finally, we describe some constructions of designs: we give an upper bound on the size of the smallest weighted unitary $t$-design in $\U(d)$, and we catalogue some $t$-designs that arise from finite groups.
\end{abstract}



\section{Introduction}
\label{sec:intro}

Recently, problems in quantum information theory \cite{Dankert06,Gross07,Scott08} have led to the study of a variation of spherical $t$-designs in which the points of the design are elements of the unitary group rather than points on a unit sphere. 

Let $X$ be a finite subset of $\U(d)$, the group of $d \times d$ unitary matrices. Then $X$ is called a \textsl{unitary $t$-design} if
\begin{equation}
\frac{1}{|X|} \sum_{U \in X} U^{\otimes t} \otimes (U^*)^{\otimes t} = \int_{\U(d)} U^{\otimes t} \otimes (U^*)^{\otimes t}\, dU,
\label{eqn:udesign}
\end{equation}
where $dU$ denotes the unit Haar measure (the unique invariant measure on $\U(d)$ normalized so that $\int_{\U(d)} dU = 1$). 
In this paper, we find lower bounds on the size of a unitary $t$-design by observing that, as implied by (\ref{eqn:udesign}), the space linearly spanned by the matrices $U^{\otimes r} \otimes (U^*)^{\otimes s}$ over all choices $U\in\U(d)$ is the same as that under the restriction $U\in X$ whenever $r+s=t$. 
In the process, we find that there is a duality between unitary designs and \defn{unitary codes}, which are finite subsets of the unitary group in which few inner product values occur between elements. We also give upper bounds on the size of a unitary code.

In order to verify directly that a set of unitary matrices forms a $t$-design, the RHS of \eqref{eqn:udesign} must be evaluated explicitly; this has been done by Collins \cite{Collins03} and Collins and \'Sniady \cite{Collins06}. In particular, $X$ is a unitary $1$-design if and only if
\[
\frac{1}{|X|} \sum_{U \in X} U \otimes U^* = \int_{\U(d)} U\otimes U^* \,dU = \frac{P_{(12)}}{d},
\]
where $P_{(12)}$ maps $u_1 \otimes u_2$ to $u_2 \otimes u_1$; $X$ is a unitary $2$-design if and only if
\[
\frac{1}{|X|} \sum_{U \in X} U^{\otimes 2} \otimes (U^*)^{\otimes 2}
= \frac{P_{(13)(24)} + P_{(14)(23)}}{d^2-1} - \frac{P_{(1423)}+P_{(1324)}}{d(d^2-1)},
\]
where here $P_{(13)(24)}$, for example, is the representation of the permutation $(13)(24)$ (written in cycle notation) in $(\cx^d)^{\otimes 4}$, and maps $u_1 \otimes u_2 \otimes u_3 \otimes u_4$ to $u_3 \otimes u_4 \otimes u_1 \otimes u_2$. For an evaluation of the integral for $t > 2$, see \cite{Scott08}. 


There is a second characterization of unitary $t$-designs, originally based on a bound of Welch \cite{Welch74}, which is easier to compute. See \cite[Theorem 5.4]{Scott08} for a proof of the following.

\begin{theorem}
\label{thm:innerproductchar}
For any finite $X \sbs \U(d)$,
\begin{equation}
\frac{1}{|X|^2}\sum_{U,V \in X} \abs{\tr(U^*V)}^{2t} \geq \int_{\U(d)} \abs{\tr(U)}^{2t} \, dU,
\label{eqn:gamma}
\end{equation}
with equality if and only if $X$ is a $t$-design.
\end{theorem}

The RHS of \eqref{eqn:gamma} can also be evaluated explicitly: it is the number of permutations of $\{1,\ldots,t\}$ that have no increasing subsequence of length greater than $d$. For example, if $d \geq t$, then the RHS is $t!$. See Diaconis and Shahshahani \cite{Diaconis94} and Rains \cite{Rains98} for details.

Let $\Hom(r,s)$ denote the polynomials that are homogeneous of degree $r$ in the matrix entries of $U \in \U(d)$ and homogeneous of degree $s$ in the entries of $U^*$. Then a third characterization is the following: $X$ is a $t$-design if, for every $f \in \Hom(t,t)$,
\[
\frac{1}{|X|} \sum_{U \in X} f(U) = \int_{\U(d)} f(U) \, dU.
\]

Every $t$-design is also a $(t-1)$-design. To see this, note that the constant function $U \mapsto \tr(U^*U)/d = 1$ is in $\Hom(1,1)$. It follows that every $f \in \Hom(t-1,t-1)$ can be embedded in $\Hom(t,t)$ by multiplying by the constant function $1$. As with spherical $t$-designs, we are interested in finding $t$-designs of minimal size. For applications in quantum information theory, $2$-designs are currently generating the most interest.


The outline of the paper is as follows. In Section \ref{sec:unitreps}, we discuss irreducible representations of the unitary group, which are needed to describe the results in later sections. In Section \ref{sec:descodes}, we introduce unitary designs and their dual notion, unitary codes. We give lower bounds on the size of a design in terms of its strength and upper bounds on the size a code in terms of its degree; we call these bounds the \defn{absolute bounds}, as they closely resemble the absolute bounds for spherical $t$-designs and codes found by Delsarte, Goethals, and Seidel \cite{Delsarte75,Delsarte77}. In Section \ref{sec:zonal}, we introduce zonal orthogonal polynomials for $\U(d)$, which are used in Section \ref{sec:relbounds} to give more precise bounds when distances in a design or code are specified; these bounds are called \defn{relative bounds}. In Section \ref{sec:weighted}, we discuss weighted $t$-designs: there is an upper bound on the size of the smallest weighted $t$-design in $\U(d)$. Finally, in Section \ref{sec:constructions}, we catalogue constructions of designs from finite groups.

\section{Unitary Representations}
\label{sec:unitreps}

We begin by describing the irreducible representations of $\U(d)$. In fact, it will suffice to describe the irreducible representations $\GL(d,\cx)$. It is clear that every representation of $\GL(d,\cx)$ restricts to a representation of $\U(d)$; however, the converse is also true. Call a representation $(\rho,V)$ of $\GL(d,\cx)$ \defn{algebraic} if the matrix entries of $\rho(M)$, $M \in \GL(d,\cx)$, are polynomials in the matrix entries $M_{ij}$ and $\det(M)^{-1}$. The following theorem is given by Bump \cite[Theorem 38.3]{Bump04}.
\begin{theorem}
Every finite-dimensional representation of\/ $\U(d)$ extends uniquely to an algebraic representation of\/ $\GL(d,\cx)$. 
\end{theorem}

So, while many of the results quoted below were originally given for $\GL(d,\cx)$, they apply equally well to $\U(d)$. 

The irreducible representations of $\U(d)$ have been studied extensively in Lie theory and much is known of them. In particular, they may be indexed by nonincreasing integer sequences of length $d$, $\mu = (\mu_1,\ldots,\mu_d)$, where $\mu_i \in \ints$ and $\mu_i \geq \mu_{i+1}$. Weyl's character formula then gives the dimension $d_\mu := \dim V_\mu$ of each irreducible representation $(\rho_\mu,V_\mu)$. For a proof of the following theorem, see \cite[Theorem 25.5]{Bump04} or \cite[Theorem 7.32 \& Exercise 7.15]{Sepanski07}. 

\begin{theorem}
The irreducible representations of\/ $\U(d)$ may be indexed by nonincreasing length-$d$ integer sequences: $\mu = (\mu_1,\ldots,\mu_d)$. If $(\rho_\mu,V_\mu)$ is the representation indexed by $\mu$, then its dimension is
\begin{equation}
d_\mu = \prod_{1 \leq i < j \leq d} \frac{\mu_i - \mu_j + j-i}{j-i}.
\label{eqn:weylchar}
\end{equation}
\end{theorem}

By way of example, the standard representation $\rho_\mu(U)=U$ is indexed by $\mu = (1,0,\ldots,0)$, which gives
\[
d_{(1,0,\ldots,0)} = d.
\]
There is generally more than one irreducible representation with the same dimension. 

Given a nonincreasing integer sequence $\mu$, let $\abs{\mu}$ denote the sum of the entries of $\mu$. If $\mu$ is a partition of $k$, then $\abs{\mu} = k$. Also, let $\mu_+$ denote the subsequence of $\mu$ of positive integers. For example, if $\mu := (1,1,0,-1)$, then $\mu_+ = (1,1)$. The decomposition of $(\cx^d)^{\otimes r} \otimes (\cx^{d*})^{\otimes s}$ into irreducible $\GL(d,\cx)$ representations has been described by Stembridge \cite{Stembridge87, Stembridge89}, and this decomposition also applies to $\U(d)$. In particular, the results of Stembridge imply the following (see also Benkart~{\it et al}. \cite[Theorem 1.3]{Benkart94}):

\begin{theorem}
The irreducible representations of\/ $\U(d)$ which occur in $(\cx^d)^{\otimes r} \otimes (\cx^{d*})^{\otimes s}$ are precisely those indexed by nonincreasing, length-$d$ integer sequences $\mu$ such that
\[
\abs{\mu} = r-s, \qquad \abs{\mu_+} \leq r.
\]
\end{theorem}

Both Stembridge and Benkart~{\it et al}.\ provide more information about the multiplicities of the representations; however, we do not require that information here. 

Next, we explain how the dimension of the span of the matrices of a representation relate to the dimension of the representation itself. 

\begin{lemma}
Let $(\rho_\mu,V_\mu)$ be an irreducible representation of\/ $\U(d)$. Then 
\[
\dim(\spn\{ \rho_\mu(U): U \in \U(d)\})={d_\mu}^2. 
\]
\end{lemma}
\begin{proof}
Let $M_\mu := \spn \{ \rho_\mu(U) : U \in \U(d)\}$. We may identify each matrix entry $(i,j)$ with a function $f_{ij}: \U(d) \rightarrow \cx$ such that $f_{ij}(U) = \rho_\mu(U)_{ij}$; therefore, we may also identify $M_\mu$ with the space of matrix coefficient functions. Clearly $\dim M_\mu \leq {d_\mu}^2$, since each matrix in $M_\mu$ is a $d_\mu \times d_\mu$ matrix. To see the dimension is exactly ${d_\mu}^2$, it suffices to show that the functions $f_{ij}$ are linearly independent as functions on $\U(d)$. This follows from the Schur orthogonality relations (see Bump \cite[Theorem 2.4]{Bump04} for example), which state that $f_{ij}$ and $f_{kl}$ are orthogonal for $(i,j) \neq (k,l)$.
\end{proof}

\begin{lemma}
Let $(\rho,V)$ be a representation of\/ $\U(d)$ and let 
\[
V = \bigoplus_\mu m_\mu V_\mu
\]
be the irreducible decomposition of $V$, so that each $V_\mu$ occurs with multiplicity $m_\mu > 0$. Then 
\[
\dim(\spn\{ \rho(U): U \in \U(d)\})=\sum_{\mu} {d_\mu}^2. 
\]
\end{lemma}

\begin{proof}
First, we claim that multiplicities do not play a role in the dimension of $M = \spn\{ \rho(U): U \in \U(d)\}$. For, suppose the irreducible component $V_\mu$ occurs $m_\mu$ times. Then the block diagonalized component of $\rho(U)$ corresponding to $V_\mu$ has the form
\[
\left( \begin{matrix} \rho_\mu(U) & 0 & 0 \\ 0 & \ddots & 0 \\ 0 & 0 & \rho_\mu(U) \end{matrix} \right) = \rho_\mu(U) \otimes I_{m_\mu},
\]
and the vector space spanned by $\{ \rho_\mu(U) \otimes I_{m_\mu}: U \in \U(d)\}$ is isomorphic to the space spanned by $\{ \rho_\mu(U): U \in \U(d)\}$. Next, we claim that if $V_\mu$ and $V_\nu$ are nonisomorphic irreducible representations, then every matrix coefficient function from $\rho_\mu(U)$ is orthogonal to every matrix coefficient function from $\rho_\nu(U)$. This is also a consequence of the Schur orthogonality relations: see Bump, \cite[Theorem 2.3]{Bump04}. 
\end{proof}

Combining these results we obtain the following closed form for the dimension of the space spanned by all representation matrices $\rho(U)=U^{\otimes r} \otimes (U^*)^{\otimes s}$. For later convenience, this is phrased in terms of the dual space of homogeneous polynomials of degree $r$ in the matrix entries of $U \in \U(d)$ and degree $s$ in the entries of $U^*$, which we denote $\Hom(r,s)$ or, in the following, $\HomU{d}{r}{s}$ to be explicit. 

\begin{theorem}
For positive integers $d$, $r$, and $s$, 
\[
\dim(\HomU{d}{r}{s}) = \sum_{\substack{\abs{\mu} = r-s \\ \abs{\mu_+} \leq r}} {d_{\mu}}^2,
\]
where the sum is over nonincreasing, length-$d$ integer sequences $\mu$, and
\[
d_\mu = \prod_{1 \leq i < j \leq d} \frac{\mu_i - \mu_j + j-i}{j-i}.
\]
\end{theorem}

For convenience set $\dimHomU{d}{r}{s} := \dim(\HomU{d}{r}{s})$ in the following. Now,
by way of example, suppose $d = 3$ and $r = s = 2$. Then
\begin{align*}
\dimHomU{3}{2}{2} & = d_{(0,0,0)}^2 + d_{(1,0,-1)}^2 + d_{(2,0,-2)}^2 + d_{(2,-1,-1)}^2 + d_{(1,1,-2)}^2 \\
& = 1^2 + 8^2 + 27^2 + 10^2 + 10^2 \\
& = 994. 
\end{align*}
It should be clear that $\dimHomU{d}{s}{r}=\dimHomU{d}{r}{s}$. The cases we are most interested in are $s=r$ and $s=r-1$. For $r\leq 3$, these dimensions are
\begin{align*}
\dimHomU{d}{1}{0} & = d^2,  &\\
\dimHomU{d}{1}{1} & = d^4-2d^2+2,  &\\
\dimHomU{d}{2}{1} & = d^2(d^4-3d^2+6)/2, & (d\geq 2) \\
\dimHomU{d}{2}{2} & = (d^8-6d^6+25d^4-28d^2+16)/4, & (d\geq 3) \\
\dimHomU{3}{3}{2} & = 2835,  &\\ 
\dimHomU{d}{3}{2} & = d^2(d^8-8d^6+47d^4-88d^2+84)/12, & (d\geq 4) \\
\dimHomU{3}{3}{3} & = 7540, \qquad \dimHomU{4}{3}{3} = 265879, & \\
\dimHomU{d}{3}{3} & = (d^{12}-12d^{10}+103d^8-378d^6+778d^4-600d^2+252)/36. & (d\geq 5) 
\end{align*} 

Additionally, it is easy to show that
\[
\dimHomU{2}{r}{s}  = \binom{r+s+3}{3}.
\]
Furthermore, since the number of independent homogeneous polynomials of degree $k$ in $n$ variables is $\binom{n+k-1}{k}$, and in $\U(d)$ we have $d^2$ variables, we also have 
\[ 
\dimHomU{d}{r}{0} = \binom{d^2+r-1}{r}.
\] 
Finally, since $\HomU{d}{r}{s}$ can be embedded into $\Hom(\cx^{d^2},r,s)$ as a subspace, the dimension of the latter provides an upper bound to the dimension of the former:
\[
\dimHomU{d}{r}{s}\leq \binom{d^2+r-1}{r}\binom{d^2+s-1}{s},
\]
from which we can conclude that $\dimHomU{d}{r}{s}=O(d^{2(r+s)})$ for fixed $r$ and $s$, or, $\dimHomU{d}{r}{s}=O((rs)^{d^2-1})$ for fixed $d$.


\section{Absolute bounds} 
\label{sec:descodes}

In this section we use the dimension of $\HomU{d}{r}{s}$ to give an upper bound on the size a unitary design, as well as an analogous lower bound on the size of a unitary code. These bounds are \defn{absolute}, meaning they depend only on the strength of the design or the degree of the code, rather than the distances that occur in the subset. Recall that $X$ is a \defn{unitary $t$-design} if for every $f \in \Hom(t,t) = \HomU{d}{t}{t}$,
\[
\frac{1}{|X|} \sum_{U \in X} f(U) = \int_{\U(d)} f(U) \, dU.
\]

The inner product for functions $f$ and $g$ on $\U(d)$ is the average value of $\overline{f}g$:
\[
\ip{f}{g} := \int_{\U(d)} \overline{f(U)}g(U) \, dU.
\]
Similarly, we use $\ip{f}{g}_X$ to denote the average value of $\overline{f}g$ over $X$, for any finite subset $X \sbs \U(d)$: this is an inner product for functions on $X$. It follows that $X$ is a unitary $t$-design if and only if
\[
\ip{1}{f} = \ip{1}{f}_X
\]
for all $f \in \Hom(t,t)$.

The \defn{absolute bound} for designs is the following:

\begin{theorem}
\label{thm:absdesbound}
If $X\subseteq\U(d)$ is a $t$-design, then
\[
|X| \geq \dimHomU{d}{\lceil t/2 \rceil}{\lfloor t/2 \rfloor}.
\]
\end{theorem}

\begin{proof}
Let $S_1,\ldots, S_N$ be an orthonormal basis for $\Hom(\lceil t/2 \rceil,\lfloor t/2 \rfloor)$. Then $\overline{S_i}S_j$ is in $\Hom(t,t)$. Since $X$ is a $t$-design,
\[
\ip{S_i}{S_j} = \ip{1}{\overline{S_i}S_j} = \ip{1}{\overline{S_i}S_j}_X = \ip{S_i}{S_j}_X.
\]
So, the polynomials $S_i:X \rightarrow \cx$ are orthogonal (and therefore independent) as functions on $X$. The space of functions on $X$ has dimension $|X|$, which is at least the number of independent elements $N$. 
\end{proof}

For $t = 2$, the bound $|X| \geq \dimHomU{d}{1}{1} = d^4 - 2d^2 + 2$ was first found by Gross, Audenaert, and Eisert \cite{Gross07}.
If equality holds in Theorem~\ref{thm:absdesbound}, then $\Hom(\lceil t/2 \rceil,\lfloor t/2 \rfloor)$ spans the space of functions on $X$. 
At the end of Section \ref{sec:relbounds} we will make further comments on this case.

The proof of Theorem~\ref{thm:absdesbound} in fact shows $|X| \geq \dimHomU{d}{r}{s}$ for any $t$-design $X$, whenever $r+s=t$; in particular, $|X| \geq \dimHomU{d}{t}{0} = \binom{d^2+t-1}{t}$, which is $\Omega(d^{2t})$ for fixed $t$, or, $\Omega(t^{d^2-1})$ for fixed $d$. Constructions meeting the latter asymptotic lower bound are known for $\U(2)$~\cite{Scott08}.

Define a finite subset $X \sbs \U(d)$ to be a \defn{unitary code of degree $s$} or a \defn{unitary $s$-distance set} if $\abs{\tr(U^*M)}^2$ takes only $s$ distinct values for all $U \neq M$ in $X$.

\begin{theorem}
\label{thm:abscodebound}
If $X\subseteq\U(d)$ is an $s$-distance set, then
\[
|X| \leq \dimHomU{d}{s}{s}.
\]
Moreover if some $U$ and $M$ in $X$ are orthogonal, then
\[
|X| \leq \dimHomU{d}{s}{s-1}.
\]
\end{theorem}

\begin{proof}
Fix $U \in X$. The function $M \mapsto \abs{\tr(U^*M)}^2$ is in $\Hom(1,1)$. If $\al_1,\ldots,\al_s$ denote the nontrivial values of $\abs{\tr(U^*M)}^2$ that occur in $X$, then define an annihilator
\[
A_U(M) := \prod_{i=1}^s (\abs{\tr(U^*M)}^2 - \al_i\tr(M^*M)/d).
\]
Then $A_U$ is in $\Hom(s,s)$, and $A_U(M) = 0$ for every $M \in X$ except $U$ (and $A_U(U)$ is nonzero). Therefore the set of functions $\{A_U: U \in X\}$ are linearly independent in $\Hom(s,s)$, and there are $|X|$ such functions.

In the case when $\al_1 = 0$, consider 
\[
A_U(M) := \tr(U^*M)\prod_{i=2}^s (\abs{\tr(U^*M)}^2 - \al_i\tr(M^*M)/d)
\]
in $\Hom(s,s-1)$ instead.
\end{proof}

\section{Zonal orthogonal polynomials}
\label{sec:zonal}

Let $\chi_\mu$ denote the character for an irreducible representation $(\rho_\mu,V_\mu)$. Then define a \defn{zonal orthogonal polynomial} as follows, for $U$ and $M$ in $\U(d)$:
\[
Z_{\mu,U}(M) := d_\mu \chi_\mu(U^*M).
\]
It is clear that $Z_{\mu,U}(M)$ depends only on $U^*M$; in fact, $Z_{\mu,U}(M)$ depends only on the eigenvalues of $U^*M$, since the characters $\chi_\mu(U^*M)$ are symmetric functions of the eigenvalues of $U^*M$. The functions are called zonal for this reason: they depend only on the ``zone", or spectrum, of the matrices. (They are called polynomials because they are polynomial in the eigenvalues of $U^*M$ and their conjugates, and they are called orthogonal because $Z_{\mu,U}$ and $Z_{\nu,M}$ are orthogonal for $\mu \neq \nu$.) This concept of ``zone" is more general than that of spherical codes and designs \cite[Chapter 15]{Godsil93}, where the zonal orthogonal polynomials are univariate. 

These polynomials have a number of special properties which will be used in the next section to derive bounds for unitary codes and designs. The most useful is following:

\begin{lemma}
\label{lem:orthorel}
Let $M_\mu$ denote the space of matrix coefficient functions on $V_\mu$. Then for any $p \in M_\mu$,
\[
\ip{Z_{\mu,U}}{p} = p(U).
\]
\end{lemma}

\begin{proof}
Recall that by the Schur orthogonality relations, the functions $f_{ij}(M) := \sqrt{d_\mu} \rho_\mu(M)_{ij}$ form an orthonormal basis for $M_\mu$. Since
\begin{align*}
Z_{\mu,U}(M) & = d_\mu \tr(\rho_\mu(U)^*\rho_\mu(M)) \\
& = d_\mu \sum_{ij} \conj{\rho_\mu(U)_{ij}}\rho_\mu(M)_{ij} \\
& = \sum_{ij} \overline{f_{ij}(U)}f_{ij}(M),
\end{align*}
it follows that $Z_{\mu,U} = \sum_{ij} \overline{f_{ij}(U)}f_{ij}$. Thus $Z_{\mu,U}$ is in $M_\mu$. Moreover, for any $p \in M_\mu$, 
\begin{align*}
\ip{Z_{\mu,U}}{p} & = \ip{\sum_{ij} \overline{f_{ij}(U)} f_{ij}}{p} \\
& = \sum_{ij}\ip{f_{ij}}{p} f_{ij}(U) \\
& = p(U),
\end{align*}
where in the last line we have written $p$ in terms of the basis $\{f_{ij}\}$.
\end{proof}

It follows from Lemma \ref{lem:orthorel} that any polynomial in $M_\mu$ which is orthogonal to every $Z_{\mu,U}$ must be identically zero. Therefore the set $\{Z_{\mu,U} : U \in \U(d)\}$ spans $M_\mu$, and we can use zonal orthogonal polynomials to characterize designs.

\begin{corollary}
\label{cor:chartdes}
A finite subset $X \sbs \U(d)$ is a $t$-design if and only if 
\[
\sum_{M \in X} Z_{\mu,U}(M) = 0
\]
for every $\mu \neq (0,\ldots,0)$ that occurs in $\Hom(t,t)$ (i.e. $\abs{\mu} = 0$, $\abs{\mu_+} \in [1,t]$).
\end{corollary}

\begin{proof}
$X$ is a $t$-design if and only if $\ip{1}{f}_X = \ip{1}{f}$ for every $f \in \Hom(t,t)$. Since $\{Z_{\mu,U}: U \in \U(d),\abs{\mu} = 0,\abs{\mu_+} \leq t\}$ spans $\Hom(t,t)$, we see that $X$ is a $t$-design if and only if $\ip{1}{Z_{\mu,U}}_X = \ip{1}{Z_{\mu,U}}$ for all $U$ and $\mu$. But $M_\mu$ is orthogonal to $M_{0,\ldots,0}$ for $\mu \neq (0,\ldots,0)$, so $\ip{1}{Z_{\mu,U}} = 0$. 
\end{proof}

We can write down a closed form of $Z_\mu$ as a symmetric polynomial in $d$ variables, so that if $\{\la_i\}$ is the spectrum of $\La = U^*M$, 
\[
Z_{\mu,U}(M) = Z_\mu(\La) = Z_\mu(\la_1,\ldots,\la_d).
\]
The following is the content of \cite[Theorem 38.2 \& Proposition 38.2]{Bump04}.

\begin{theorem}
Let $V_\mu$ be the irreducible representation of $\U(d)$ indexed by nonincreasing integer sequence $\mu$. If $\mu_d = 0$, then the character of $V_\mu$ is
\[
\chi_\mu(\La) = s_\mu(\la_1,\ldots,\la_d),
\]
where $s_\mu$ is the Schur polynomial, and $\{\la_1,\ldots,\la_d\}$ are the eigenvalues of $\La$. If $\mu_d \neq 0$, then define 
\[
\mu' = (\mu_1-\mu_d,\ldots,\mu_{d-1}-\mu_d,0).
\]
Then the character of $V_\mu$ is
\[
\chi_\mu(\La) = \det(\La)^{\mu_d}\chi_{\mu'}(\La).
\]
\end{theorem}

By way of example, the character of $V_{(1,0,\ldots,0)}$, the standard representation, is $\chi_{(1,0,\ldots,0)}(\La)= s_{(1,0,\ldots,0)}(\la_1,\ldots,\la_d) = \sum_i \la_i = \tr(\La)$, where $\la_i$ are the eigenvalues of $\La$. The dimension of $V_{(1,0,\ldots,0)}$ is $d$, so the first zonal orthogonal polynomial is
\[
Z_{1,U}(M) = d\tr(\La) = d\tr(U^*M).
\]

\begin{table}{\footnotesize\[
\begin{array}{|c|c|}\hline
\mu & \chi_\mu(\la_i) \\
\hline
(0,\ldots,0) & 1 \\
(1,0,\ldots,0) & \sum_i \la_i \\
(1,0,\ldots,0,-1) & \sum_{i,j}\frac{\la_i}{\la_j} - 1 \\
(2,0,\ldots,0,-1) & \sum \frac{\la_i\la_j}{2\la_k} + \sum \frac{\la_i^2}{2\la_j} - \sum \la_i \\
(1,1,0,\ldots,0,-1) & \sum \frac{\la_i\la_j}{2\la_k} - \sum \frac{\la_i^2}{2\la_j} - \sum \la_i \\
(2,0,\ldots,0,-2) & \sum \frac{\la_i\la_j}{4\la_k\la_l} + \sum \frac{\la_i^2}{4\la_k\la_l} + \sum \frac{\la_i\la_j}{4\la_k^2} + \sum \frac{\la_i^2}{4\la_k^2}
- \sum \frac{\la_i}{\la_j} \\
(2,0,\ldots,0,-1,-1) & \sum \frac{\la_i\la_j}{4\la_k\la_l} + \sum \frac{\la_i^2}{4\la_k\la_l} - \sum \frac{\la_i\la_j}{4\la_k^2} - \sum \frac{\la_i^2}{4\la_k^2}
- \sum \frac{\la_i}{\la_j} + 1 \\
(1,1,0,\ldots,0,-2) & \sum \frac{\la_i\la_j}{4\la_k\la_l} - \sum \frac{\la_i^2}{4\la_k\la_l} + \sum \frac{\la_i\la_j}{4\la_k^2} - \sum \frac{\la_i^2}{4\la_k^2}
- \sum \frac{\la_i}{\la_j} + 1 \\
(1,1,0,\ldots,0,-1,-1) & \sum \frac{\la_i\la_j}{4\la_k\la_l} - \sum \frac{\la_i^2}{4\la_k\la_l} - \sum \frac{\la_i\la_j}{4\la_k^2} + \sum \frac{\la_i^2}{4\la_k^2} - \sum \frac{\la_i}{\la_j} \\
\hline\hline
\mu & Z_{\mu,U}(M) \\
\hline
(0,\ldots,0) & 1 \\
(1,0,\ldots,0) & d\tr(\La) \\
(1,0,\ldots,0,-1) & (d^2-1)\bigl(\abs{\tr(\La)}^2-1\bigr) \\
(2,0,\ldots,0,-1) & \frac{d(d - 1)(d + 2)}{2}\Bigl(\frac{1}{2}[\tr(\La)^2 + \tr(\La^2)]\tr(\La^*) - \tr(\La)\Bigr) \\
(1,1,0,\ldots,0,-1) & \frac{d(d + 1)(d - 2)}{2}\Bigl(\frac{1}{2}[\tr(\La)^2 - \tr(\La^2)]\tr(\La^*) - \tr(\La)\Bigr) \\
(2,0,\ldots,0,-2) & \frac{d^2(d - 1)(d + 3)}{4}\left(\frac{1}{4}\!\abs{\tr(\La)^2 + \tr(\La^2)}^2 - \abs{\tr(\La)}^2 \right)\\
(2,0,\ldots,0,-1,-1) & \frac{(d^2 - 1)(d^2 - 4)}{4}\left(\frac{1}{4}[\tr(\La)^2+\tr(\La^2)][\overline{\tr(\La)^2-\tr(\La^2)}] - \abs{\tr(\La)}^2 + 1 \right)\\
(1,1,0,\ldots,0,-2) & \frac{(d^2 - 1)(d^2 - 4)}{4}\left(\frac{1}{4}[\tr(\La)^2-\tr(\La^2)][\overline{\tr(\La)^2+\tr(\La^2)}] - \abs{\tr(\La)}^2 + 1 \right)\\
(1,1,0,\ldots,0,-1,-1) & \frac{d^2(d + 1)(d - 3)}{4}\left(\frac{1}{4}\!\abs{\tr(\La)^2 - \tr(\La^2)}^2 - \abs{\tr(\La)}^2 \right)\\ 
\hline\end{array}
\]}
\caption{The first few characters and zonal orthogonal polynomials, for sufficiently large $d$, given in terms of $\La = U^*M$ and its eigenvalues $\lambda_i$.}
\label{tbl:zonals}\end{table}

\section{Relative bounds}
\label{sec:relbounds}

With zonal orthogonal polynomials in place, we can get a bound on the size of a unitary code which depends on the particular inner product values that occur. These bounds are called \defn{relative bounds}, in contrast with the absolute bound in Theorem \ref{thm:abscodebound} that depends only on the number of distinct inner product values. The proof, which in some sense uses linear programming, mimics the relative bounds of Delsarte \cite{Delsarte73} for codes in association schemes and the bounds of Delsarte, Goethals, and Seidel \cite{Delsarte75} for spherical codes. 

\begin{theorem}
\label{thm:relcodebound}
Let $F$ be a function of the eigenvalues of a matrix and $X$ a finite subset of $\U(d)$ such that $F(U^*M)$ is real and non-positive for every $U \neq M$ in $X$. If $F = \sum_{\abs{\mu}=0} c_\mu Z_\mu$, where $c_\mu \geq 0$ for every $\mu$ and $c_{(0,\ldots,0)} > 0$, then
\begin{equation}
|X| \leq F(I)/c_{(0,\ldots,0)}.
\label{eqn:relcodebound1}
\end{equation}
Moreover, if $c_\mu > 0$ for every $\abs{\mu_+} \leq t$, then equality holds in equation \eqref{eqn:relcodebound1} if and only if $X$ is a $t$-design.
\end{theorem}

\begin{proof}
Since $F(U^*M) \leq 0$ for every $U \neq M$, summing over all $M \in X$, we have
\[
\sum_{M \in X} F(U^*M) \leq F(U^*U) = F(I).
\]
Then averaging over all $U \in X$,
\begin{align*}
F(I) & \geq \frac{1}{|X|}\sum_{U,M \in X} F(U^*M) \\
& = \frac{1}{|X|} \sum_\mu c_\mu \sum_{U,M \in X} Z_\mu(U^*M). 
\end{align*}
Using Lemma \ref{lem:orthorel}, the inner sum is non-negative for $\mu \neq (0,\ldots,0)$:
\[
\sum_{U,M \in X} Z_{\mu,U}(M) = \sum_{U,M \in X} \ip{Z_{\mu,M}}{Z_{\mu,U}} = \ip{\sum_{U \in X}Z_{\mu,U}}{\sum_{U \in X}Z_{\mu,U}} \geq 0.
\]
On the other hand $Z_{(0,\ldots,0)}(U^*M) = 1$ for all $U$ and $M$. Therefore,
\begin{align*}
F(I) & \geq \frac{1}{|X|} c_{(0,\ldots,0)} \sum_{U,M \in X} 1 \\
& = c_{(0,\ldots,0)} |X|, 
\end{align*} 
and equation \eqref{eqn:relcodebound1} follows. If equality holds and $c_\mu > 0$, then 
\[
\sum_{U \in X}Z_{\mu,U}(M) = 0
\]
for each $\mu \neq (0,\ldots,0)$, which implies
\[
\sum_{U \in X}Z_{\mu,M}(U) = 0,
\]
and so $X$ is a $t$-design by Corollary \ref{cor:chartdes}. 
\end{proof}

Similarly, one can show that if $\tr(M^*U)F(U^*M) \leq 0$ for every $U \neq M$ in $X$, and $F = \sum_{\abs{\mu}=1} c_\mu Z_\mu$ such that $c_\mu \geq 0$ and $c_{(1,0,\ldots,0)} > 0$, then
\[
|X| \leq F(I)/c_{(1,0,\ldots,0)}.
\]
The next two corollaries are examples of how Theorem \ref{thm:relcodebound} can be applied.

\begin{corollary}
\label{cor:rel1distbound}
If $X$ is a unitary $1$-distance set with $\abs{\tr(U^*M)}^2 = \al < 1$ for all $U \neq M$ in $X$, then
\[
|X| \leq \frac{d^2-\al}{1-\al}.
\]
Equality holds if and only if $X$ is a $1$-design.
\end{corollary}

\begin{proof}
Take $F$ to be the annihilator of $\al$, $F(\la) = \abs{\sum_i \la_i}^2 - \al$, so that $F(U^*M) = 0$ for every $U,M \in X$. Since $Z_{(0,\ldots,0)} = 1$ and $Z_{(1,0,\ldots,0,-1)} = (d^2-1)(\abs{\sum_i \la_i}^2 - 1)$ we may write
\[
F = \frac{1}{d^2-1}Z_{(1,0,\ldots,0,-1)} + (1-\al)Z_{(0,\ldots,0)}
\]
and so by Theorem \ref{thm:relcodebound}
\[
|X| \leq \frac{F(I)}{c_0} = \frac{d^2-\al}{1-\al}.
\]
\end{proof}

More generally, for any $X \subset \U(d)$ let 
\[
\al := \max_{U \neq M \in X} \abs{\tr(U^*M)}^2.
\]
If $\al < 1$, then the argument of Corollary \ref{cor:rel1distbound} shows that $|X| \leq \frac{d^2-\al}{1-\al}$, with equality if and only if $X$ is an equiangular $1$-design. In particular, when $\al = \frac{d^2-2}{d^2-1}$, we find that $|X| \leq d^4 - 2d^2+2$ (the bound in Theorem \ref{thm:absdesbound} for $2$-designs).


\begin{corollary}
\label{cor:rel2distbound}
Let $X$ be a unitary $2$-distance set with $\abs{\tr(U^*M)}^2 \in \{\al,\be\}$ for all $U \neq M$ in $X$, such that
\begin{align*}
\al + \be & \leq 4, \\
\al + \be & < \al\be + 2.
\end{align*}
Then
\[
|X| \leq \frac{(d^2-\al)(d^2-\be)}{\al\be - \al - \be + 2}.
\]
Equality holds if and only if $X$ is a $2$-design.
\end{corollary}

\begin{proof}
Take $F(\la) = (\abs{\sum_i \la_i}^2 - \al)(\abs{\sum_i \la_i}^2 - \be)$, the annihilator of $\al$ and $\be$, and write it in terms of zonal orthogonal polynomials. Define
\begin{align*}
F_2(\la) & := \frac{Z_{(2,0,\ldots,0,-2)}}{D_{(2,0,\ldots,0,-2)}} + \frac{Z_{(2,0,\ldots,0,-1,-1)}}{D_{(2,0,\ldots,0,-1,-1)}} + \frac{Z_{(1,1,0,\ldots,0,-2)}}{D_{(1,1,0,\ldots,0,-2)}} + \frac{Z_{(1,1,0,\ldots,0,-1,-1)}}{D_{(1,1,0,\ldots,0,-1,-1)}} \\
& = \sum_{i,j,k,l} \frac{\la_i\la_j}{\la_k\la_l} - 4\sum_{i,k} \frac{\la_i}{\la_k} + 2.
\end{align*}
Therefore,
\begin{align*}
F(\la) & = \sum_{i,j,k,l} \frac{\la_i\la_j}{\la_k\la_l} - (\al+\be)\sum_{i,k} \frac{\la_i}{\la_k} + \al\be \\
& = F_2(\la) + \frac{(4-\al-\be)Z_{(1,0\ldots,0,-1)}(\la)}{D_{(1,0\ldots,0,-1)}} + (\al\be - \al - \be +2).
\end{align*}
Applying Theorem \ref{thm:relcodebound}, 
\[
|X| \leq \frac{F(I)}{c_{(1,0\ldots,0)}} = \frac{(d^2-\al)(d^2-\be)}{\al\be - \al - \be +2}.
\]
\end{proof}

In particular, if $X$ is the union of a set of mutually unbiased unitary-operator bases \cite[Theorem 5.6]{Scott08}, then $\{\al,\be\} = \{0,1\}$ and 
\[
|X| \leq d^2(d^2-1),
\]
with equality if and only if $X$ is a $2$-design.

Every $s$-distance set in $\U(d)$ is also (up to scaling) an $s$-distance set in the complex space $\cx^{d^2}$. When $s = 1$, the bound in Corollary \ref{cor:rel1distbound} coincides with the bound for complex lines due to Delsarte, Goethals and Seidel \cite[Tables 1 \& 2]{Delsarte75}. However, the upper bounds in Theorem \ref{thm:abscodebound} and Corollary \ref{cor:rel2distbound} for sets in $\U(d)$ are smaller than the upper bounds for sets in $\cx^{d^2}$.

There is also a (less useful) relative bound for designs.
\begin{theorem}
\label{thm:reldesbound}
Let $F$ be a symmetric function of the eigenvalues of a matrix, suppose $X$ is a unitary $t$-design such that $F(U^*M) \geq 0$ for every $U \neq M$ in $X$. If $F = \sum_{\abs{\mu}=0} c_\mu Z_\mu$, where $c_{(0,\ldots,0)} > 0$ and $c_\mu \leq 0$ for every $\abs{\mu_+} > t$, then
\[
|X| \geq F(I)/c_{(0,\ldots,0)}.
\]
Equality holds if and only if $F(U^*M) = 0$ for every $U \neq M$ in $X$ and $X$ is a $\abs{\mu_+}$-design for every $c_\mu > 0$.
\end{theorem}



Again we give two examples.

\begin{corollary}
\label{cor:reldesbound1}
If $\al < 1$ and $X$ is a unitary $1$-design such that $\abs{\tr(U^*M)}^2 \geq \al$ for all $U \neq M$ in $X$, then 
\[
|X| \geq \frac{d^2-\al}{1-\al},
\]
with equality if and only if $\abs{\tr(U^*M)}^2 = \al$ for all $U$ and $M$.
\end{corollary}

\begin{proof}
Take $F$ to be the annihilator of $\al$, as in Corollary \ref{cor:rel1distbound}. Then apply Theorem \ref{thm:reldesbound}. 
\end{proof}

\begin{corollary}
\label{cor:reldesbound2}
Let $\al$ and $\be$ satisfy
\begin{align*}
\al & < \be, \\
\al + \be & \leq 4, \\
\al + \be & < \al\be + 2.
\end{align*}
If $X$ is a $2$-design such that $\abs{\tr(U^*M)}^2 \leq \al$ or $\abs{\tr(U^*M)}^2 \geq \be$ for all $U$ and $M$ in $X$, then 
\[
|X| \geq \frac{(d^2-\al)(d^2-\be)}{\al\be - \al - \be + 2},
\]
with equality if and only if $\abs{\tr(U^*M)}^2 \in \{\al,\be\}$ for all $U$ and $M$. 
\end{corollary}

In particular, if $X$ is a $2$-design such that $\abs{\tr(U^*M)}^2 = 0$ or $\abs{\tr(U^*M)}^2 \geq \be$ for all $U$ and $M$ in $X$, then 
\[
|X| \geq \frac{d^2(d^2-\be)}{2-\be},
\]
with equality if and only if $\abs{\tr(U^*M)}^2 \in \{0,\be\}$ for all $U$ and $M$.

A unitary $t$-design $X$ is \defn{tight} if it satisfies the bound in Theorem \ref{thm:absdesbound} with equality:
\[
|X| = \dimHomU{d}{\lceil t/2 \rceil}{\lfloor t/2 \rfloor}.
\]
Other types of tight designs, such as tight spherical designs or tight complex projective designs, play an important role in algebraic combinatorics \cite{Delsarte75} and quantum information \cite{Renes04}. They are less significant here as they appear to be very rare: numerical searches have not revealed any tight $t$-designs at all for $t > 1$. Nevertheless, if they do exist, Theorem \ref{thm:reldesbound} gives some necessary conditions for their structure. 

\begin{corollary}
For $r = \lceil t/2 \rceil$ and $s = \lfloor t/2 \rfloor$, define
\[
F_t(\la) := 
\sum_{\substack{\abs{\mu} = r-s \\ \abs{\mu_+} \leq r}} Z_\mu(\la).
\]
If $X$ is a $t$-design of minimal size $\dimHomU{d}{r}{s}$, then $F_t$ is an annihilator of $X$, in the sense that
$F_t(U^*M) = 0$ for every $U \neq M$ in $X$.
\end{corollary}

\begin{proof}
Since $F_t(\la) \in \Hom(r,s)$, it follows that $G_t(\la) = \abs{F_t(\la)}^2$ is in $\Hom(t,t)$. Moreover, $G_t(U^*M) \geq 0$ for every $U \neq M$ in $X$, and relations in the Schur polynomials show that when $G_t$ is written in the form $\sum_\mu c_\mu Z_\mu(\la)$, we have $c_{(0,\ldots,0)} = F_t(I)$. Applying Theorem \ref{thm:reldesbound}, we find that 
\[
|X| \geq \frac{G_t(I)}{c_{(0,\ldots,0)}} = F_t(I) = \dimHomU{d}{r}{s}.
\]
Equality holds only if $G_t(U^*M) = 0$ for every $U \neq M$ in $X$. 
\end{proof}

A tight unitary $1$-design in $\U(d)$ has size $d^2$, and the matrices of the design are orthogonal. Such designs exist and are called \defn{unitary operator bases} \cite{Klappenecker03}. A tight $2$-design has size $d^4-2d^2+2$ and is equiangular with angle $\abs{\tr(U^*M)} = \frac{d^2-2}{d^2-1}$ (see also \cite[Theorem 5.5]{Scott08}). A tight unitary $3$-design $X$ has size $d^2(d^4-3d^2+6)/2$, and for every $U \neq M$ in $X$, $\La = U^*M$ satisfies
\[
\tfrac{1}{2}[(d^2 - 2)\tr(\La)^2 + d\tr(\La^2)]\tr(\La^*) = (d^2-3)\tr(\La).
\]

\section{Weighted designs}
\label{sec:weighted}

Let $X$ be a finite subset of $\U(d)$ and let $w: X \rightarrow \re$ be a positive weight function (i.e., $w>0$, $\sum_{U \in X} w(U) = 1$). Then $(X,w)$ is a \textsl{weighted unitary $t$-design} if
\begin{equation}
\sum_{U \in X} w(U) \, U^{\otimes t} \otimes (U^*)^{\otimes t} = \int_{\U(d)} U^{\otimes t} \otimes (U^*)^{\otimes t} \, dU.
\label{eqn:wudesign}
\end{equation}
Every $t$-design $X$ is a weighted $t$-design with weight function $w(x) := 1/|X|$. 

Many results about $t$-designs apply equally well to weighted $t$-designs. In particular, from the proof of Theorem \ref{thm:absdesbound} we get the following:

\begin{corollary}
\label{cor:abswdesbound}
If $(X,w)$ is a weighted $t$-design in $\U(d)$, then
\[
|X| \geq \dimHomU{d}{\lceil t/2 \rceil}{\lfloor t/2 \rfloor}.
\]
Equality holds only if $X$ is a $t$-design.
\end{corollary}



The fact that equality holds only for unweighted designs follows from an argument of Levenshtein \cite[Theorem 4.3]{Levenshtein98}. The bounds in Theorem \ref{thm:reldesbound} and Corollaries \ref{cor:reldesbound1} and \ref{cor:reldesbound2} also hold for weighted designs, with equality only if the designs are unweighted. 

For many applications of classical design theory or spherical designs, weighted designs are just as useful as unweighted designs. Moreover, weighted designs are easier to find. Here, we give an upper bound on the size of smallest weighted $t$-design, using arguments of Godsil \cite[Theorem 14.10.1]{Godsil93} (see also \cite{Godsil86}) and de la Harpe and Pache \cite[Proposition 2.6]{Harpe05}.

\begin{theorem}
\label{thm:desupperbound}
For any $t$ and $d$, there exists a weighted $t$-design $(X,w)$ in $\U(d)$ with 
\[
|X| \leq \dimHomU{d}{t}{t}.
\]
\end{theorem}

\begin{proof}
Given a nonincreasing integer sequence $\mu = (\mu_1,\dots,\mu_d)$, define $\overline{\mu} := (-\mu_d,\dots,-\mu_1)$, and note that $Z_{\overline{\mu},U}(M) = \overline{Z_{\mu,U}(M)}$. Also define $f_U:=(Z_{\mu_1,U},\dots,Z_{\mu_N,U})$ and $\tilde{f}_U(M):=\Re f_U(M)\oplus\Im f_U(M)$, where we take $\{\mu_1,\dots,\mu_N\}=\{\mu: 0<|\mu_+|\leq t,|\mu|=0\}$. Since $|\overline{\mu}_+|=|\mu_+|$ for $|\mu|=0$, the polynomial $Z_{\overline{\mu},U}(M)=\overline{Z_{\mu,U}(M)}$ appears as an entry in $f_U$ whenever $Z_{\mu,U}(M)$ does (i.e.\ ${V_\mu}^*$ appears in the irreducible decomposition of $V^{\otimes t} \otimes (V^*)^{\otimes t}$ with the same multiplicity as $V_\mu$). Now define $\Omega_X:=\{f_U: U \in X\}$ and $\tilde{\Omega}_X:=\{\tilde{f}_U: U \in X\}$ for $X\subseteq\U(d)$. Since $\overline{Z_{\mu,U}}$ appears in $f_U$ whenever $Z_{\mu,U}$ does, it follows that 
\begin{align*}
\dim_{\re}(\spn_{\re}\tilde{\Omega}_{\U(d)}) = \dim_{\cx}(\spn_{\cx}\Omega_{\U(d)}) 
     = \!\sum_{\substack{|\mu|=0 \\ 0<|\mu_+|\leq t}}\!{d_\mu}^2 
     = \dimHomU{d}{t}{t}-1.
\end{align*}


Now, for any $\mu\neq(0,\dots,0)$,
\begin{align*}
\int_{\U(d)} Z_{\mu,U}(M)dU = \int_{\U(d)}\overline{Z_{\mu,M}(U)} \, dU = \ip{Z_{\mu,M}}{1} =0,
\end{align*}
and thus 
\[
\int_{\U(d)} f_U (M) \, dU = (0,\dots,0),
\] 
which means the zero function $0$ is in the convex hull of $\tilde{\Omega}_{\U(d)}$; see \cite[Lemma 2.5]{Harpe05}. By Carath\'{e}odory's theorem, there exists a finite subset $X\subseteq\U(d)$ for which $0$ is also in the convex hull of $\tilde{\Omega}_X$. That is, there are positive real weights $w(U)$, with $\sum_{U\in X}w(U)=1$, such that
\begin{align*}
(0,\dots,0) = \sum_{U\in X} w(U) f_U (M) =\sum_{U\in X} w(U) (Z_{\mu_1,U}(M),\dots,Z_{\mu_N,U}(M)), 
\end{align*}
and thus,
\begin{align*}
(0,\dots,0) &= \sum_{U\in X} w(U) (\overline{Z_{\mu_1,U}(M)},\dots,\overline{Z_{\mu_N,U}(M)}) \\
&= \sum_{U\in X} w(U) (Z_{\mu_1,M}(U),\dots,Z_{\mu_N,M}(U)),
\end{align*}
independent of $M$. This means $(X,w)$ is a weighted $t$-design (by Corollary~\ref{cor:chartdes} in the weighted case). Moreover, again by Carath\'{e}odory's theorem, $X$ can be chosen with cardinality $|X|\leq \dim_{\re}(\spn_{\re}\tilde{\Omega}_{\U(d)})+1= \dimHomU{d}{t}{t}$.
\end{proof}

In fact, since $\U(d)$ is connected, the bound can be improved by one: $|X| \leq \dimHomU{d}{t}{t}-1$; see de la Harpe and Pache \cite[Proposition 2.7]{Harpe05}.

\section{Constructions}
\label{sec:constructions}

Some constructions of 2-designs were given by Dankert, Cleve, Emerson and Livine \cite{Dankert06} and Gross, Audenaert and Eisert \cite{Gross07}. The argument of Seymour and Zaslavsky \cite{Seymour84} shows that $\U(d)$ $t$-designs exist for every $t$ and $d$ \cite[Corollary 5.3]{Scott08}, and the result of the previous section bounds the size of the smallest weighted $t$-design to $O(d^{4t})$. Since $\U(2)$ $t$-designs are equivalent to $\re P^3$ $t$-designs~\cite{Scott08}, the interesting dimensions are $d>2$, in which case, all of the known constructions with $t > 1$ come from groups. 

Group designs were defined by Gross {\it et al\/} \cite{Gross07} as the images of unitary representations of finite groups. We make good use of Theorem~\ref{thm:innerproductchar} in this section and thus restate it here for the special case of group designs. Let $\rho$ be a unitary representation of a finite group $G$. Since $\rho(g)^*\rho(h)=\rho(g^{-1}h)$ for all elements $g$ and $h$, we can test whether $X=\{\rho(g):g\in G\}$ is a unitary $t$-design in terms of the character $\chi=\tr\rho$ alone: 
\begin{corollary}
\label{cor:groupdesign}
Let $G$ be a finite group and $\rho:G\rightarrow\U(d)$ a representation with character $\chi$. Then $X=\{\rho(g):g\in G\}$ is a unitary $t$-design if and only if
\[
\frac{1}{|G|}\sum_{g \in G} \abs{\chi(g)}^{2t} = \int_{\U(d)} \abs{\tr(U)}^{2t} \, dU.
\]
\end{corollary}

\begin{table}{\footnotesize
\begin{tabular}{|r|r|r|r|c|c|}\hline
$d$  & $t$  & $\dimHomU{d}{\lceil t/2 \rceil}{\lfloor t/2 \rfloor}$ & $|X|=|H|$ & $H=G/Z(G)$ & $G$ \{$\chi$ no.\} \\\hline\hline 
$q$ & 2 & $q^4-2q^2+2$ & $q^5-q^3$      &  $\ff_q^2\rtimes\mathrm{Sp}(2,q)$   &     \\\hline
2    & 2      & 10      & 12              & $\ff_2^2\rtimes H'_\textrm{C2} \cong A_4$   & SL(2,3) \{4\}  \\
2    & 3  & 20    & 24                 & $\ff_2^2\rtimes\mathrm{Sp}(2,2)\cong S_4$    & GL(2,3) \{4\}  \\
2    & 5   & 56       & 60             & $A_5$   &  SL(2,5) \{2\}  \\\hline
3    & 2    & 65      & 72                & $\ff_3^2\rtimes H'_\textrm{C3}$  &  {\tt 2\^{}3.L3(2)} \{2\} \\
3    & 3    & 270    & 360             & $A_6$   & {\tt 3.A6} \{8\} \\\hline
4    & 3    & 1 712   & 2 520          & $A_7$   &  {\tt 6.A7}  \{10\} \\\hline
5    & 2        & 577   & 600             & $\ff_5^2\rtimes H'_\textrm{C5}$   &  {\tt 5\^{}1+2.2A4} \{9\} \\\hline
6    & 2     & 1 226   & 2 520            & $A_7$   &  {\tt 6.A7}  \{31\} \\
6    & 3  & 21 492   & 40 320          &    &  {\tt 6.L3(4).2\_1} \{49\}  \\\hline
7    & 2     & 2 305    & 2 352           & $\ff_7^2\rtimes H'_\textrm{C7}$   &    \\\hline
8    & 2     & 3 970    & 20 160          &    &   {\tt 4\_1.L3(4)} \{19\} \\\hline
9    & 2      & 6 401  & 12 960           & $\ff_3^4\rtimes H'_\textrm{GAE}$   &     \\\hline
10   & 2     & 9 802    & 95 040          &    &  {\tt 2.M12} \{16\} \\\hline
11   & 2     & 14 401   & 14 520          &  $\ff_{11}^2\rtimes H'_\textrm{C11}$  &    \\\hline
12   & 3 & 1 462 320 & 448 345 497 600  &    &  {\tt 6.Suz} \{153\}  \\\hline
14   & 2     & 38 026   & 87 360          &    & {\tt Sz(8).3} \{4\} \\\hline
18   & 3  & 16 849 620 & 50 232 960     &    &  {\tt 3.J3} \{22\} \\\hline
21   & 2     & 193 601   & 9 196 830 720  &    &  {\tt 3.U6(2)}  \{47\} \\\hline
26   & 2     & 455 626  & 17 971 200      &    &  {\tt 2F4(2)\char"0D} \{2\} \\\hline
28   & 2   & 613 090  & 145 926 144 000   &    & {\tt 2.Ru} \{37\} \\\hline
45   & 2    & 4 096 577  & 10 200 960      &    &  {\tt M23} \{3\} \\\hline
342  & 2   & 13 680 343 370  & 460 815 505 920  &     & {\tt 3.ON} \{31\} \\\hline
 1333 & 2    & $3.157..\times 10^{12}$ & $8.677..\times 10^{19}$  &    &  {\tt J4} \{2\} \\\hline
\end{tabular}}
\caption{The size $|X|$ of the smallest known group design for each $t$ and $d$. The lower bound (Theorem~\ref{thm:absdesbound}) is included for comparison. Group names in the {\tt typewriter} font are those used by the GAP character library CTblLib 1.1.3; e.g.\ the GAP command {\tt chi:=Irr(CharacterTable("3.A6"))[8]} gives an irreducible character for a $\U(3)$ 3-design, which can be confirmed by checking that {\tt Degree(chi)} and {\tt Norm(chi\^{}3)} output $d$ and $t!$ respectively. The Clifford designs, i.e.\ those with $H\cong\ff_q^{2n}\rtimes H'$ for some $H'\leq\mathrm{Sp}(2n,q)$ ($d=q^n$), were described by Gross {\it et al\/} \cite{Gross07}, where $H'_\textrm{GAE}\leq\mathrm{Sp}(4,3)$ was also given. The subgroups  $H'_{\textrm{C}d}\leq\mathrm{Sp}(2,d)$ that generate 2-designs of size $|X|=d^2(d^2-1)$ ($d=2,3,5,7,11$) were found by Chau~\cite{Chau05}, and are defined in the text.}
\label{tbl:groupdesigns}\end{table}

Thus, whenever $X$ represents a group, we can confirm $X$ to be a $t$-design by simply checking that $\ip{1}{f}_X=\ip{1}{f}$ for the single polynomial $f(U)=\abs{\tr(U)}^{2t}$; in general, a basis for $\Hom(t,t)$ must be checked.

Note that $X$ is a unitary 1-design if and only if $\rho$ is irreducible: if $\bigoplus_\mu(\rho_\mu\otimes I_{m_\mu})$ is an irreducible decomposition of $\rho$, then the multiplicities satisfy
\[
\sum_\mu m_\mu^2 = \frac{1}{|G|}\sum_{g \in G} \abs{\chi(g)}^2 = \int_{\U(d)} \abs{\tr(U)}^{2} \, dU=1
\]
by Schur orthogonality. We can therefore restrict our search for group designs to \emph{irreducible} representations of finite groups. By Schur's lemma, we then have $\rho(g)\propto I$ for all $g\in Z(G)$, and the size of the design can be reduced to $|G|/|Z(G)|$ by ignoring the $|Z(G)|$ different phase factors. Alternatively, the central quotient $H=G/Z(G)$ defines a unitary $t$-design through the irreducible projective unitary representation $\pi:H\rightarrow\PU(d)$ defined by ignoring the $|Z(G)|$ different phases: $\pi(g):=[\rho(g)]$. Note that all group designs of a fixed size can now be found, in principle, by calculating a Schur covering group $C$ of each group $H$ of that size. This is because every projective representation of $H$ can be lifted to an ordinary representation of $C$. Unfortunately, while exhaustive, this approach is computationally expensive and no additional designs were discovered in this way beyond those described next.

Following Gross {\it et al\/} \cite{Gross07}, an alternative approach is to harvest unitary designs straight from the known character tables of finite groups. The smallest known group design for each $t$ and $d$ is listed in Table~\ref{tbl:groupdesigns}. Of these, the so-called {\em Clifford designs\/} deserve special mention: for any $H'\leq\mathrm{Sp}(2n,q)$ that acts transitively on $\ff_q^{2n}\setminus\{0\}$ (and therefore has order $|H'|=k(q^{2n}-1)$ for some integer $k$), there exists a projective unitary representation of $H=\ff_q^{2n}\rtimes H'$ in dimension $d=q^n$ that defines a group 2-design of size $|X|=|H|=d^2|H'|=kd^2(d^2-1)$ \cite{Gross07}. Taking $H=\ff_d^{2}\rtimes\mathrm{Sp}(2,d)$ generates a 2-design in all prime-power dimensions. The resulting size is $|X|=d^3(d^2-1)$, however, which exceeds the lower bound by a factor of $d$ asymptotically. In general, deciding whether there exists a family of 2-designs of size $O(d^4)$, in arbitrarily large dimensions, is an important open problem. It is known that within the class of Clifford designs improvements can be made in some small dimensions: Chau~\cite{Chau05} has given subgroups of $\mathrm{Sp}(2,d)$ with minimal order $|H'|=d^2-1$ that generate 2-designs of size $|X|=d^2(d^2-1)$ in dimensions $d=2,3,5,7,11$ (unfortunately, such optimal $H'$ were proven not to exist in other dimensions); additionally, Gross {\it et al\/} \cite{Gross07} discovered a subgroup of $\mathrm{Sp}(4,3)$ generating a 2-design of size $|X|=2d^2(d^2-1)$ in dimension $d=9$. Explicitly, the Chau subgroups are $H'_\textrm{C2}=\big\langle\big[\begin{smallmatrix} 0&1\\ 1&1 \end{smallmatrix} \big]\big\rangle$, $H'_\textrm{C3}=\big\langle\big[\begin{smallmatrix} 1&1\\ 1&2 \end{smallmatrix} \big],\big[\begin{smallmatrix} 1&2\\ 2&2 \end{smallmatrix} \big]\big\rangle$, $H'_\textrm{C5}=\big\langle\big[\begin{smallmatrix} 2&0\\ 0&3 \end{smallmatrix} \big],\big[\begin{smallmatrix} 1&2\\ 1&3 \end{smallmatrix} \big]\big\rangle$, $H'_\textrm{C7}=\big\langle\big[\begin{smallmatrix} 2&0\\ 0&4 \end{smallmatrix} \big],\big[\begin{smallmatrix} 1&2\\ 1&3 \end{smallmatrix} \big]\big\rangle$, and $H'_\textrm{C11}=\big\langle\big[\begin{smallmatrix} 2&0\\ 0&6 \end{smallmatrix} \big],\big[\begin{smallmatrix} 1&1\\ 1&2 \end{smallmatrix} \big]\big\rangle$, whilst the subgroup discovered by Gross {\it et al\/} is given in Table II of \cite{Gross07}.

Finally, we remark that, although there are currently no known nontrivial examples of weighted $t$-designs for $d>2$, the minimal 2-design in $\U(2)$ is necessarily weighted: it is known that no tight $t$-designs exist for $\U(2)$ when $t>1$ and, in fact, the 12-point group $\U(2)$ 2-design in Table~\ref{tbl:groupdesigns} is minimal for unweighted designs; but there exists a weighted 11-point $\U(2)$ 2-design~\cite{Scott08}. 

\section{Acknowledgements}

AR is supported by NSERC, MITACS, PIMS, iCORE, and the University of Calgary Department of Mathematics and Statistics Postdoctoral Program. AJS is supported by ARC grant CE0348250 and thanks the Institute for Quantum Information Science for their hospitality.


\end{document}